\documentclass{article}
\usepackage{pinlabel}
\usepackage{graphicx}
\usepackage{amsmath}
\usepackage{amssymb}
\usepackage{amsthm}
\usepackage{multirow}
\usepackage{color}
\usepackage{soul}  
\usepackage{hyperref}
\usepackage[english]{babel}

\input epsf.tex

\usepackage{epstopdf}   
\newtheorem{theorem}{Theorem}
\newtheorem{proposition}[theorem]{Proposition}
\newtheorem{corollary}[theorem]{Corollary}

\newtheorem{example}{Example}
\newtheorem{remark}[theorem]{Remark}

\renewenvironment{proof}[1][Proof]{\textit{#1.} }
{\hfill \rule{0.5em}{0.5em}}

\begin{document}


\title{A geometric description of the extreme Khovanov cohomology}
\author{J. Gonz\'alez-Meneses, P. M. G. Manch\'on and M. Silvero \footnote{All authors partially supported by MTM2013-44233-P and FEDER.}}
\maketitle

\begin{abstract}
We prove that the hypothetical extreme Khovanov cohomology of a link is the cohomology of the independence simplicial complex of its Lando graph. We also provide a family of knots having as many non-trivial extreme Khovanov cohomology modules as desired, that is, examples of $H$-thick knots which are as far of being $H$-thin as desired.
\end{abstract}

\section{Introduction} \label{IntroductionSection}
The Khovanov cohomology of knots and links was introduced by Mikhail Khovanov at the end of last century (see~\cite{Khovanov}) and nicely explained by Bar-Natan in~\cite{Bar}. In~\cite{Viro} Viro interpreted it in terms of enhanced states of diagrams. Using Viro's point of view, in this paper we will prove that the hypothetical extreme Khovanov cohomology of a link coincides with the cohomology of the independence simplicial complex of its Lando graph.

The Lando graph (\cite{Lando}) of a link diagram was studied by Morton and Bae in~\cite{MortonBae}, where they proved that the hypothetical extreme coefficient of the Jones polynomial coincides with a certain numerical invariant of the graph, named in this paper as its independence number. Now, on one hand the Jones polynomial can be seen as the bigraded Euler characteristic of the Khovanov cohomology. On the other hand, the formula for the independence number certainly suggests the formula of an Euler characteristic. Both ideas together have led us to a way of understanding the extreme Khovanov cohomology in terms of this graph. This is what we develop in this paper and reflect in Theorem~\ref{KeyTheorem}.

In \cite{extreme} it was proved that the independence number can take any value, hence there are links (in fact knots) with arbitrary extreme coefficients. This idea is also extended here to Khovanov cohomology, by proving that there are links (in fact knots) with an arbitrary number of non-trivial extreme Khovanov cohomology modules (see Theorem~\ref{teothick}), which are of course examples of $H$-thick knots (see \cite{patterns}). The basic piece of these examples is a link with exactly two non-trivial extreme Khovanov cohomology modules, constructed in Theorem~\ref{Thexample}. However, we think that the construction of the example is more interesting than the example itself. The example is constructed using Theorem \ref{KeyTheorem}, the Alexander duality and a construction by Jonsson \cite{Jonsson}.

There are many other interesting constructions of graphs starting with a link diagram (see for example \cite{Dasbach}), not to be confused with the Lando graph. In addition, there are other very interesting ways of trying to understand the Khovanov cohomology as the cohomology of {\it something}. For example in \cite{LipshitzSarkar} Lipshitz and Sarkar construct, in an explicit and combinatorial way, a chain bicomplex that produces the Khovanov cohomology.

The paper is organized as follows: Section~\ref{KhovanovSection} reviews the definition of the Khovanov cohomology of links, by using the notion of enhanced state of a diagram, and precise what we understand by hypothetical  extreme Khovanov complex and cohomology modules. Section~\ref{LandoSection} reviews the notion of Lando graph of a link diagram, and defines the independence simplicial complex of this graph. We will refer to the corresponding cochain complex as the Lando cochain complex, and to the corresponding cohomology modules as the Lando cohomology. Indeed, we will note that the independence number of the Lando graph is the Euler characteristic associated to the Lando cohomology. In Section~\ref{EquivalenceSection} we prove that the Lando cochain complex is a copy of the hypothetical extreme Khovanov complex, with the degrees shifted by a constant, Theorem~\ref{KeyTheorem}. In this Section there is also an example showing how to apply this theorem in order to compute Khovanov extreme cohomology in a practical way. The rest of the paper deals with applications of Theorem~\ref{KeyTheorem}. In particular, Section~\ref{grafoacomplejosection} shows how to relate Lando cohomology (and extreme Khovanov cohomology) to the homology of a simplicial complex, which allows us to give an example of a link diagram having exactly two non-trivial extreme Khovanov cohomology modules. Finally in Section~\ref{ThickKnotsSection} we use the previous example in order to give families of $H$-thick knots.


\section{Khovanov cohomology} \label{KhovanovSection}

Let $D$ be an oriented diagram of a link $L$ with $c$ crossings and writhe $w = p - n$, with $p$ and $n$ being the number of positive and negative crossings in $D$, according to the sign convention shown in Figure \ref{marcadores}. A state $s$ assigns a label, $A$ or $B$, to each crossing of $D$. Let $\mathcal{S}$ be the collection of $2^c$ possible states of $D$. For $s \in \mathcal{S}$ assigning $a(s)$ $A$-labels and $b(s)$ $B$-labels, write $\sigma = \sigma(s) = a(s) - b(s)$. The result of smoothing each crossing of $D$ according to its label following Figure \ref{marcadores} is a collection $sD$ of disjoint circles embedded in the plane together with some $A$ and $B$-chords (segments joining the circles in the site where there was a crossing). We represent $A$-chords as dark segments, and $B$-chords as light ones. See Figure \ref{trebolsuavizado}. Enhance the state $s$ with a map $e$ which associates a sign $\epsilon_i = \pm 1$ to each of the $|s|$ circles in $sD$. Unless otherwise stated, we will keep the letter $s$ for the enhanced state $(s, e)$ to avoid cumbersome notation. Write $\tau = \tau (s) = \sum_{i=1}^{|s|} \epsilon_i$, and define, for the enhanced state $s$, the integers

$$
i = i(s)=\frac{w - \sigma}{2},     \quad      j = j(s) = w + i + \tau.
$$

\begin{figure}
\centering
\includegraphics[width = 10cm]{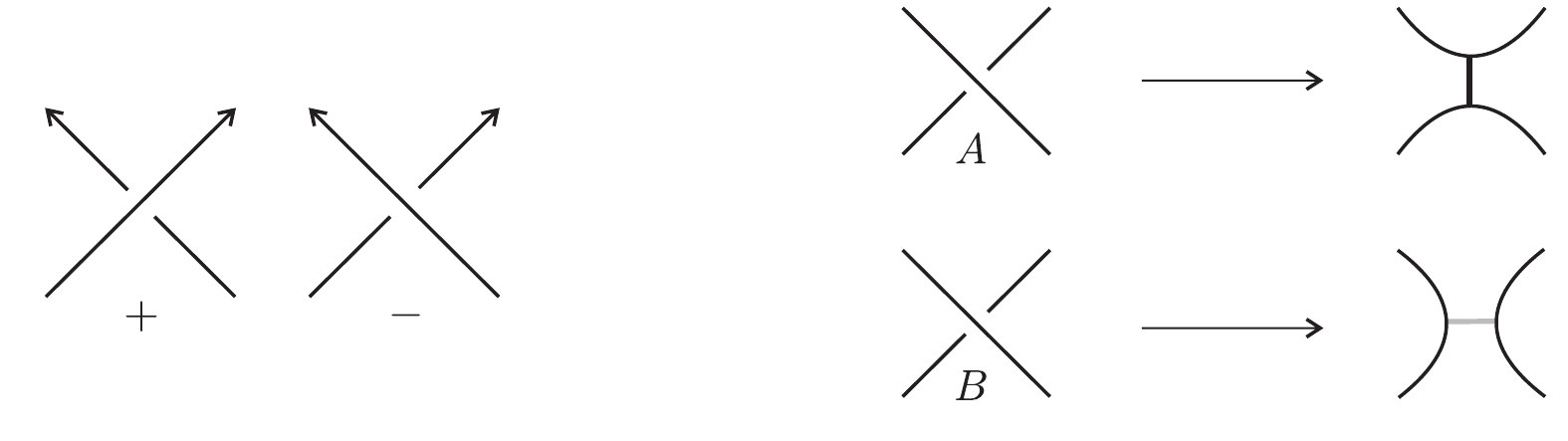}
\caption{\small{The leftmost part shows our convention of signs. The rightmost part shows the smoothing of a crossing according to its $A$ or $B$-label. $A$-chords ($B$-chords) are represented by dark (light) segments.}}
\label{marcadores}
\end{figure}

\begin{figure}
\centering
\includegraphics[width = 9.5cm]{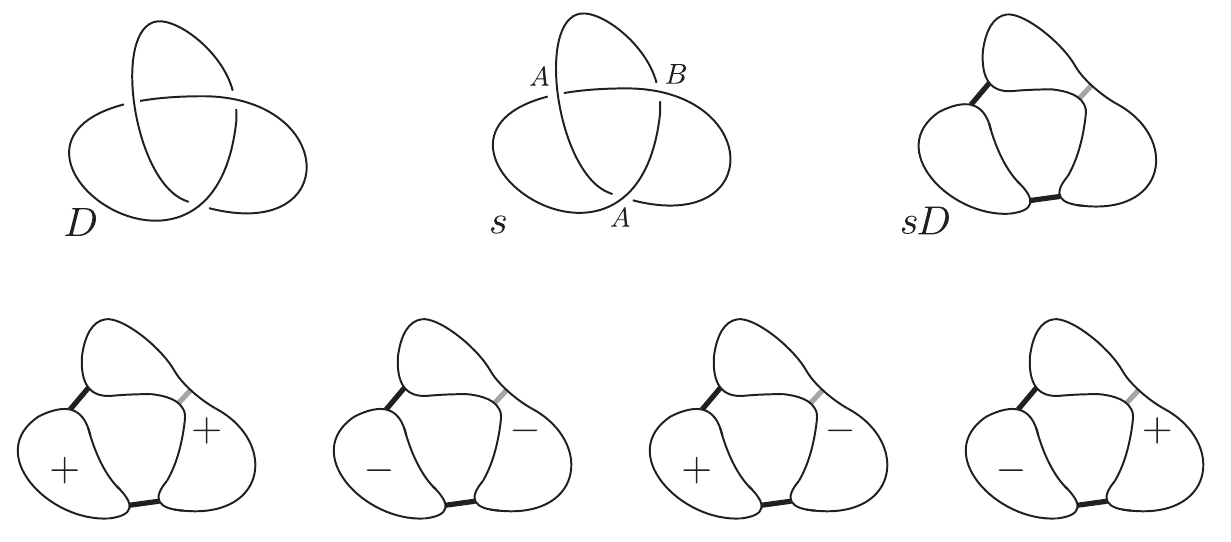}
\caption{\small{The first row shows a diagram $D$ representing $3_1$, a state $s$ and $sD$. Here $|s| = 2$. The four related enhanced states are shown in the second row. From left to right the values of $\tau$ are $2, -2, 0, 0$.}}
\label{trebolsuavizado}
\end{figure}

The enhanced state $t$ is adjacent to $s$ if the following conditions are satisfied:
\begin{enumerate}
\item $i(t) = i(s) + 1$ and $j(t) = j(s)$.
\item The labels assigned by $t$ are identical to those assigned by $s$ except in one (change) crossing $x = x(s,t)$, where $s$ assigns an $A$-label and $t$ a $B$-label.
\item The signs assigned to the common circles in $sD$ and $tD$ are equal.
\end{enumerate}

Note that the circles which are not common to $sD$ and $tD$ are those touching the crossing $x$. In fact, passing from $sD$ to $tD$ can be realized by melting two circles into one, or splitting one circle into two. The different possibilities according to the previous points are shown in Figure~\ref{posibilidades}.

\begin{figure}[h]
\centering
\includegraphics[width = 11cm]{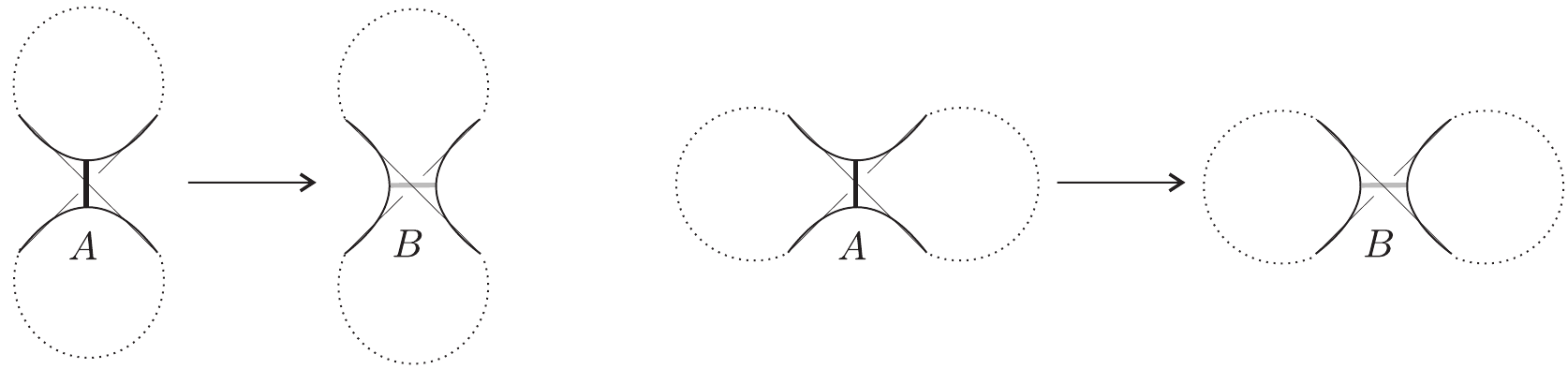}
\caption{\small{All possible enhancements when melting two circles are: $(++ \rightarrow +), \, (+- \rightarrow -), \, (-+ \rightarrow -)$. The possibilities for the splitting are: $(- \rightarrow --), \, (+ \rightarrow +-)$ or $(+ \rightarrow -+)$.}}
\label{posibilidades}
\end{figure}


Let $R$ be a commutative ring with unit. Let $C^{i,j}(D)$ be the free module over $R$ generated by the set of enhanced states $s$ of $D$ with $i(s) = i$ and $j(s) = j$. Order the crossings in $D$. Now fix an integer $j$ and consider the ascendant complex
$$
\ldots \quad \rightarrow \quad C^{i,j}(D) \quad \stackrel{d_i}{\longrightarrow} \quad C^{i+1,j}(D) \quad \longrightarrow \quad \ldots
$$
with differential $d_i(s) = \sum (s:t) t$, where $(s:t) = 0$ if $t$ is not adjacent to $s$ and otherwise $(s:t) = (-1)^k$, with $k$ being the number of $B$-labeled crossings coming after the change crossing $x$. It turns out that $d_{i+1} \circ d_i = 0$ and the corresponding cohomology modules over $R$
$$
H^{i,j}(D)=\frac{\textnormal{ker} (d_i)}{\textnormal{im}(d_{i-1})}
$$
are independent of the diagram $D$ representing the link $L$ and the ordering of its crossings, that is, these modules are link invariants. They are the Khovanov cohomology modules $H^{i,j}(L)$ of $L$ (\cite{Khovanov}, \cite{Bar}), as presented by Viro in \cite{Viro} in terms of enhanced states.

Let $j_{\min} = j_{\min}(D) = \min \{ j(s)\ / \ s \textnormal{ is a enhanced state of } D\}$. We will refer to the complex $\{ C^{i,j_{\min}}(D), d_i\}$ as the extreme Khovanov complex, and to the corresponding modules $H^{i,j_{\min}}(D)$ as the (hypothetical) extreme Khovanov modules. Indeed, there are analogous definitions for a certain $j_{\max}$.

We remark that the integers $j_{\min}$ and $j_{\max}$ depend on the diagram $D$, and may differ for two different diagrams representing the same link.


\section{Lando graph and its cohomology}\label{LandoSection}

Let $G$ be a graph. A set $\sigma$ of vertices of $G$ is said to be independent if no vertices in $\sigma$ are adjacent. The independence number of $G$ is defined to be
$$
I(G)=\sum_{\sigma} (-1)^{|\sigma |}
$$
where the sum is taken all over the independent sets of vertices of $G$. The empty set is considered as an independent set of vertices with $|\emptyset |=0$. A point has independence number 0, an hexagon 2.

Starting from a link diagram $D$ we recall the construction of its Lando graph (\cite{MortonBae}, \cite{extreme}). An $A$ or $B$-chord of $sD$ is admissible if its ends are both in the same circle of $sD$. Let $s_A$ be the state assigning $A$-labels to all the crossings of $D$. The Lando graph associated to $D$ is constructed from $s_AD$ by considering a vertex for every admissible $A$-chord, and an edge joining two vertices if the ends of the corresponding $A$-chords alternate in the same circle. Figure~\ref{ejemplolando} exhibits a diagram $D$, the corresponding $s_AD$ and its Lando graph $G_D$.

\begin{figure}
\centering
\includegraphics[width = 12cm]{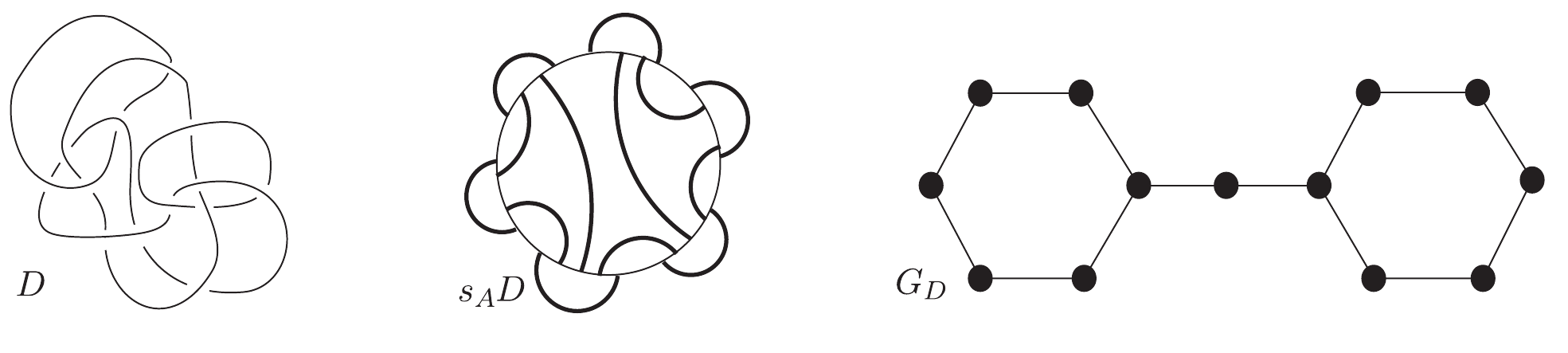}
\caption{\small{A diagram $D$ representing the trivial link with 4 components, $s_AD$ and the corresponding Lando graph $G_D$.}}
\label{ejemplolando}
\end{figure}

In \cite{MortonBae} it was proved that the coefficient of the hypothetical lowest degree monomial of the Jones polynomial of a link $L$ is, up to sign, the independence number $I(G_D)$ of the Lando graph $G_D$ of $D$, where $D$ is a diagram of $L$. In \cite{extreme} it was shown that $I(G_D)$ can take arbitrary values.

The Jones polynomial can be seen as the bigraded Euler characteristic of the Khovanov cohomology. The formula for the independence number suggests that each extreme coefficient of the Jones polynomial is the Euler characteristic of a certain cohomology given in terms of independent sets of vertices of the Lando graph. This idea has led us to a way of understanding the extreme Khovanov cohomology in terms of this graph.

Let $X_D$ be the independence simplicial complex of the graph $G_D$. By definition, the simplices $\sigma$ of $X_D$ are the independent subsets of vertices of $G_D$. Let $C^i(X_D)$ be the free module over $R$ generated by the simplices of dimension $i$, where the dimension of a simplex $\sigma$ is the number of its vertices minus one. Associated to this simplicial complex we have the (standard) cochain complex
$$
\ldots \rightarrow C^{i}(X_D) \, \stackrel{\delta_i}{\longrightarrow} \, C^{i+1}(X_D) \rightarrow \ldots
$$
with differential $\delta_i(\sigma )=\sum_v (-1)^{k} \sigma \cup \{ v\}$ where $v$ runs over the set of vertices of $G_D$ which are not adjacent to any vertex of $\sigma$, and $k = k(\sigma, v)$ is the number of vertices of $\sigma$ coming after $v$ in the predetermined order of the vertices of $G_D$. It turns out that $\delta_{i+1} \circ \delta_i=0$ and the corresponding reduced cohomology modules of $X_D$ are
$$
H^{i}(X_D)=\frac{\textnormal{ker} (\delta_i)}{\textnormal{im}(\delta_{i-1})}.
$$

We will refer to $\{ C^i(X_D), \delta_i \}$ as the Lando ascendant complex, and to $H^i(X_D)$ as the Lando cohomology modules. It is a direct algebraic exercise to check that, if $R$ is the field of the rational numbers, then $I(G_D)$ is the Euler characteristic associated to the cohomology of~$X_D$.


\section{Lando and extreme Khovanov cohomologies} \label{EquivalenceSection}

Let $D$ be an oriented link diagram and $s_A^-$ the enhanced state assigning an $A$-label to each crossing and the sign $-1$ to each circle of $s_AD$. Let
$$S_{\min} = \{ \textnormal{enhanced states $s$ } \, / \, |s| = |s_A| + b(s), \, \, \tau (s)= -|s| \}.$$

\begin{proposition}\label{PropositionSmin}
With the previous notation $j_{\min} = j(s_A^-)$ and $j(s) = j_{\min}$ if and only if $s \in S_{\min}$.
\end{proposition}

\begin{proof}
Recall that $j(s) = \frac{3w - \sigma}{2} + \tau$ with $\sigma = a(s) - b(s)$ and $\tau (s) = \sum_{i=1}^{|s|} \epsilon_i$, where $\epsilon_i$ is the sign (+1 or -1) associated to the circle $c_i$ in $sD$.

Given a diagram $D$ let $s$ be an enhanced state associating a positive sign to at least one of the circles in $sD$. Then $j(s) \neq j_{\min}$, as the state given by associating negative signs to every circle in $sD$ has a smaller value of $j$. Hence, all states $s$ realizing $j_{\min}$ assign $-1$ to their circles, or equivalently $\tau (s)= -|s|$ (the second condition in the definition of $S_{\min}$).

Now identify a state with the set of crossings of $D$ where the state assigns a $B$-label. Let $s = \{ y_1, \dots , y_b\}$ be an enhanced state assigning $b = b(s)$ $B$-labels (at the crossings $y_1, \dots, y_b$) and negative signs to all its circles. Consider the sequence of enhanced states
$$
s_0=s_A^-, s_1, \dots , s_b = s
$$
where $s_k = \{y_1, \ldots, y_k \}$ for $k = 1, \dots , b$, and all circles in $s_kD$ have sign $-1$.

Since $w$ is invariant and $\sigma (s_k) = \sigma (s_{k-1}) - 2$, there are two possibilities: if $|s_k|= |s_{k-1}| + 1$, then $\tau(s_k) = \tau(s_{k-1}) - 1$ and $j(s_k) = j(s_{k-1})$; if $|s_k|= |s_{k-1}| - 1$, then $\tau(s_k) = \tau(s_{k-1}) + 1$ and  $j(s_k) = j(s_{k-1}) + 2$. Note that this is independent of the ordering of the crossings.

A first consequence is that $j(s_A^-) \leq j(s)$, where $s$ was taken to be any state assigning $-1$ to all circles in $sD$, so $j(s_A^-) = j_{\min}$. A second consequence is that $j(s) = j_{\min}$ if and only if $|s_k| = |s_{k-1}| + 1$ for each $k \in \{1, \dots , b\}$, that is, if and only if $s \in S_{\min}$.
\end{proof}

There are analogous $s_B^{+}$, $j_{\max}$ and $S_{max}$, with $j(s_B^+) = j_{\max}$ and $s \in S_{max}$ if and only if $j(s) = j_{\max}$.

\begin{corollary}\label{CorollaryJmin}
Fix an oriented link diagram $D$ with $c$ crossings, $n$ negative and $p$ positive. Then $j_{\min}=c - 3n - |s_A|$ and $j_{\max} = - c + 3p + |s_B|$.
\end{corollary}

\begin{proof}
Since $w = p - n = c - 2n$ and $\sigma(s) = c - 2b(s)$ we deduce that $i(s) = b(s) - n$. In particular $i(s_A) = - n$. It follows that

$$
\begin{array}{rcl}
j_{\min} & = & j(s_A^-) \\
 & = & w + i(s_A^-) + \tau(s_A^-) \\
 & = & (c - 2n) - n - |s_A| \\
 & = & c - 3n - |s_A|.
\end{array}
$$

A similar argument works for $j_{\max}$ using $s_B^+$ instead of $s_A^-$.
\end{proof}

Recall that the vertices in the Lando graph of $D$, $G_D$, are associated to the admissible $A$-chords in $s_AD$ (the ones having both ends in the same circle of $s_AD$). Let $V_s$ be the set of vertices of $G_D$ corresponding to the crossings of $D$ to which $s$ associates a $B$-label. Note that $V_s$ can have less than $b(s)$ vertices, or even be empty.

\begin{proposition}\label{PropositionCs}
The map that assigns $V_s$ to each enhanced state $s$ defines a bijection between $S_{min}$ and the set of independent sets of vertices of $G_D$. Moreover, if $s \in S_{min}$ then the cardinal of $V_s$ is exactly $b(s)$.
\end{proposition}

\begin{proof}
Let $s = \{ y_1, \dots , y_b\}$ be an enhanced state in $S_{\min}$ with $b = b(s)$ $B$-labels (at the crossings $y_1, \dots , y_b$). Consider the sequence of enhanced states
$$
s_0=s_A^-, s_1, \dots , s_b=s
$$
where $s_k = \{y_1, \ldots, y_k \}$ for $k = 1, \dots , b$, and all circles in $s_kD$ have sign $-1$. As $s \in S_{min}$, according to the proof of Proposition~\ref{PropositionSmin} $|s_{k}| = |s_{k-1}| + 1$ for each $k \in \{1, \dots , b\}$, or equivalently, one passes from $s_{k-1}D$ to $s_kD$ by splitting one circle into two circles.

Note that the $A$-chord of $s_AD$ corresponding to the crossing $y_1$ of $D$ is admissible, since otherwise $|s_1| = |s_0|-1$ (see Figure~\ref{prueba1}). As the construction in the previous sequence does not depend on the order of the crossings, it follows that any $A$-chord of $s_AD$ corresponding to a crossing $y_i$ is admissible in $s_AD$, so $G_D$ contains its associated vertex.

\begin{figure}
\centering
\includegraphics[width = 10cm]{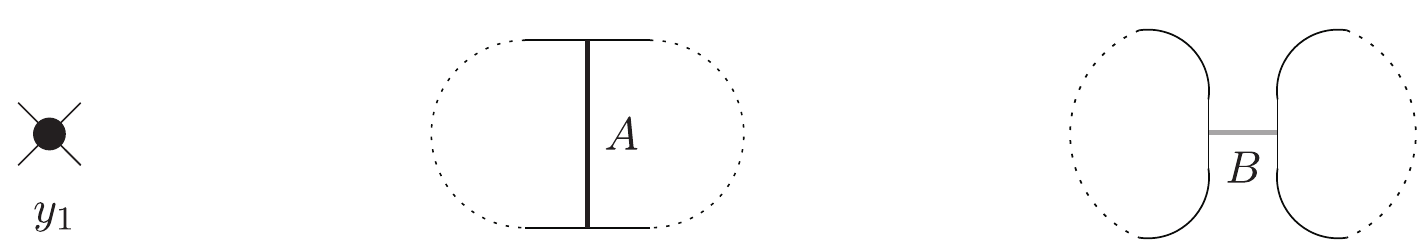}
\caption{\small{The vertex $y_1$ corresponds to a splitting from $s_AD = s_0D$ to $s_1D$.}}
\label{prueba1}
\end{figure}

Moreover there is no pair of $A$-chords in $s_AD$ corresponding to $B$-labels of $s$ with their ends alternating in the same circle, since otherwise two $B$-smoothings in these two crossings would not increase the number of circles by two, as Figure~\ref{prueba2} shows schematically. This implies that the corresponding vertices in $V_s$ are independent.

\begin{figure}
\centering
\includegraphics[width = 10cm]{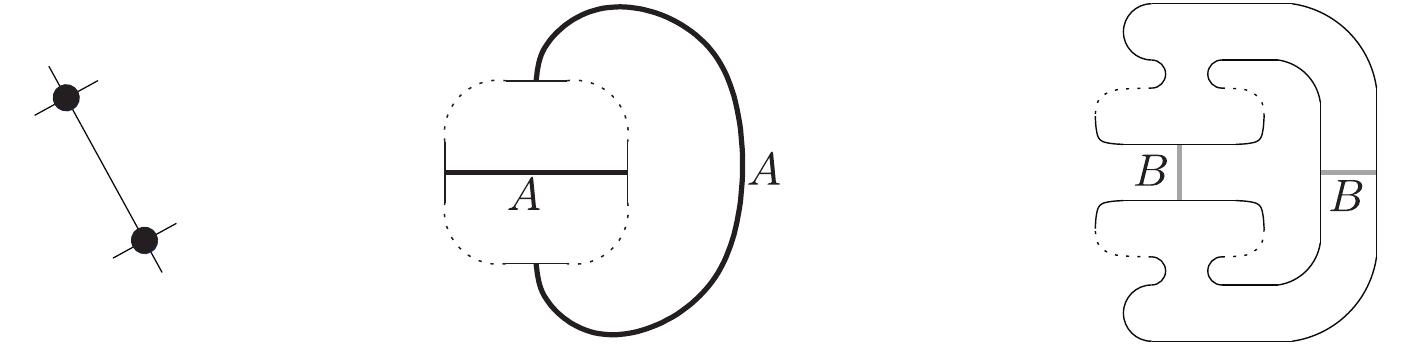}
\caption{\small{Adjacent vertices in $G_D$ correspond to $|s_A| = |s|$ when smoothing.}}
\label{prueba2}
\end{figure}

Conversely, if $C$ is an independent set of vertices of $G_D$, consider the state $s$ that assigns $B$-labels exactly in the corresponding crossings. In particular $b(s) = |C|$. Enhance this state assigning $-1$ to each circle of $sD$. Since $C$ is independent, $|s| = |s_A| + b(s)$, hence $s \in S_{min}$ as we wanted to show.
\end{proof}

The extreme Khovanov cohomology is constructed, according to Proposition~\ref{PropositionSmin}, in terms of the states in $S_{min}$. For these states the definition of adjacency given in Section~\ref{KhovanovSection} is reduced to the second condition given there. Namely, if $s, t \in S_{min}$ then $t$ is adjacent to $s$ if and only if $t$ assigns the same labels as $s$ except in one crossing $x$, where $s(x)$ and $t(x)$ are an $A$-label and a $B$-label, respectively. Now we are ready to establish the main result in this paper:

\begin{theorem}\label{KeyTheorem}
Let $L$ be an oriented link represented by a diagram $D$ having $n$ negative crossings. Let $G_D$ be the Lando graph of $D$ and let $j=j_{\min}(D)$. Then the Lando ascendant complex $\{C^i(X_D), \delta_i\}$ is a copy of the extreme Khovanov complex $\{C^{i,j}(D), d_i\}$, shifted by $n - 1$. Hence
$$
H^{i,j}(D)\approx H^{i-1+n}(X_D).
$$
\end{theorem}

\begin{proof}
According to Proposition~\ref{PropositionSmin}, the extreme Khovanov cohomology is constructed with the states in $S_{min}$. Suppose that $s \in S_{min}$ and let $V_s$ be the corresponding independent set of vertices of $G_D$. Since $i(s) = b(s) - n$ and $\dim (V_s) = |V_s| - 1 = b(s) -1$, the bijection between $S_{min}$ and the set of independent sets of vertices of $G_D$ established in Proposition~\ref{PropositionCs} provides an isomorphism
$$
C^{i,j}(D) \approx C^{i-1+n}(X_D).
$$

One just needs to show that this isomorphism respects the differential of both complexes (technically, that the assignment $s$ to $V_s$ defines a chain isomorphism). Recall that for two states $s, t \in S_{min}$, $t$ is adjacent to $s$ if and only if $t$ assigns the same labels as $s$ except in one crossing $x$, where $s(x)$ and $t(x)$ are an $A$-label and a $B$-label, respectively. Moreover, in this case only a splitting is possible at the change crossing $x$ when passing from $sD$ to $tD$, since the degree $j_{\min}$ is preserved and the degree~$i$ is increased by one, $\tau (s) = -|s|$ and $\tau (t) = -|t|$. It follows that $V_t = V_s \cup \{v_x\}$ where $v_x$ is the vertex in $G_D$ corresponding to $x$.

In addition, if we order the vertices of $G_D$ according to the order of the crossings in $D$ (hence the assignment ``vertex to crossing'' is an increasing map), we get that the number of $B$-labeled crossings of $D$ coming after the crossing $x$, is exactly the number of vertices of $V_s$ coming after the vertex $v_x$.
\end{proof}

\begin{example} \label{Example}

Consider the oriented link $L$ represented by the oriented diagram $D$ shown in the leftmost part of Figure~\ref{ejhexagon}. The Lando graph $G_D$ is the hexagon~ shown at the right hand side of Figure \ref{ejhexagon}. Number its vertices consecutively, from $1$ to~$6$. Then $C^{-1}(X_D)$, $C^{0}(X_D)$, $C^{1}(X_D)$ and $C^{2}(X_D)$ have ranks $1$, $6$, $9$ and $2$ respectively, with respective basis
$$\{ \emptyset\}, \{1, 2, 3, 4, 5, 6 \}, \{13, 14, 15, 24, 25, 26, 35, 36, 46\}, \{135, 246\},$$
the other modules being trivial (note that we write, for example, $135$ instead of $\{ 1, 3, 5\}$). The Lando ascendant complex is

$$
0 \, \longrightarrow \, C^{-1}(X_D) \, \stackrel{\delta_{-1}}{\longrightarrow} \, C^{0}(X_D) \, \stackrel{\delta_0}{\longrightarrow} \, C^{1}(X_D) \, \stackrel{\delta_1}{\longrightarrow} \, C^{2}(X_D) \, \longrightarrow \, 0,
$$

with differentials $\delta_{-1}, \, \delta_{0}$ and $\delta_{1}$ given respectively by the matrices

$$
\left( \begin{array}{c} 1 \\ 1 \\ 1 \\ 1 \\ 1 \\ 1 \end{array} \right),
\left( \begin{array}{cccccc}
1 & 0 & -1 & 0 & 0 & 0 \\
1 & 0 & 0 & -1 & 0 & 0 \\
1 & 0 & 0 & 0 & -1 & 0 \\
0 & 1 & 0 & -1 & 0 & 0 \\
0 & 1 & 0 & 0 & -1 & 0 \\
0 & 1 & 0 & 0 & 0 & -1 \\
0 & 0 & 1 & 0 & -1 & 0 \\
0 & 0 & 1 & 0 & 0 & -1 \\
0 & 0 & 0 & 1 & 0 & -1
\end{array} \right),
\left( \begin{array}{ccccccccc} 1 & 0 & -1 & 0 & 0 & 0 & 1 & 0 & 0 \\
0 & 0 & 0 & 1 & 0 & -1 & 0 & 0 & 1 \end{array}\right).
$$

Let $R$ be the field of rational numbers. The ranks of these matrices are $1$, $5$ and $2$ respectively. In particular $H^{1}(X_D)$ has dimension two as a rational vector space, being trivial the rest of Lando cohomology vector spaces.

Orient now the three components of $D$ in a counterclockwise sense. By Corollary~\ref{CorollaryJmin}, $j_{\min} = c - 3n - |s_A| = 6 - 3 \cdot 6 - 1 = -13$ and, by Theorem \ref{KeyTheorem}, the complexes are shifted by $n - 1 = 5$, hence $H^{-4,-13}(L) \approx H^1(X_D)$ is two-dimensional, being trivial the rest of extreme Khovanov cohomology vector spaces.

This example shows that, in general, for different orientations of the components of a link, we obtain the same extreme Khovanov cohomology modules ($j_{\min}$ may change), with some shifting in the index $i$.

\begin{figure}
\centering
\includegraphics[width = 10.5cm]{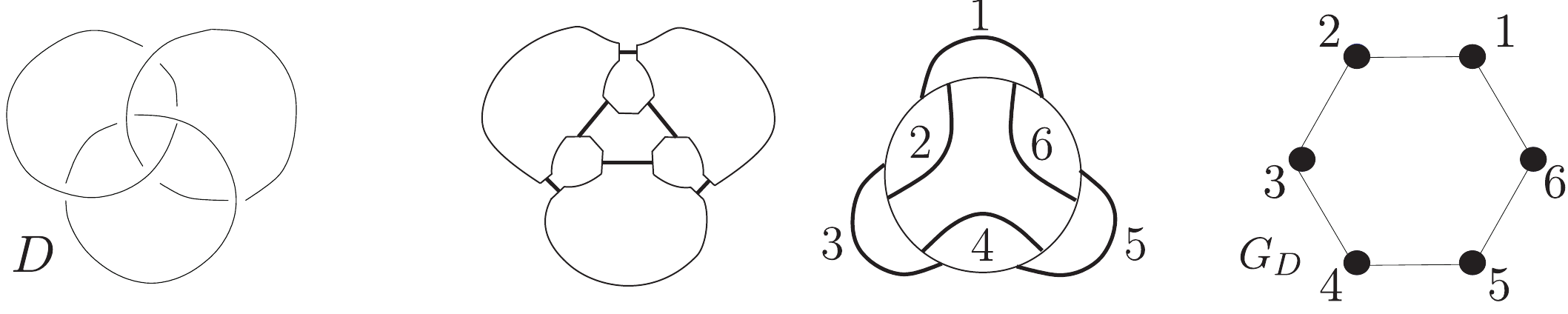}
\caption{\small{A link diagram $D$, two versions of $s_AD$ and its associated Lando graph $G_D$.}}
\label{ejhexagon}
\end{figure}

\end{example}


A complete bipartite graph $K_{r,s}$ is a graph whose vertices can be divided into two disjoint sets $V$ and $W$, having $r$ and $s$ vertices respectively, such that every edge connects a vertex in $V$ to one in $W$ and every pair of vertices $v \in V$ and $w \in W$ are connected by an edge.

\begin{corollary}\label{CorollaryKrs}
Let $L$ be an oriented link represented by a diagram $D$ having $n$ negative crossings. Let $G_D$ be the Lando graph of $D$ and let $j=j_{\min}(D)$. Then $H^{1-n,j}(L) \approx~R$ if $G_D$ is the complete bipartite graph $K_{r,s}$ and it is trivial otherwise.
\end{corollary}

\begin{proof}
According to Theorem~\ref{KeyTheorem} we just have to prove that $H^0(X_D)\approx~R$ if $G_D$ is $K_{r,s}$, and it is trivial otherwise. Let $G_D^c$ be the complementary graph of $G_D$. A remarkable observation is that any Lando graph is always a bipartite graph, and $G_D=K_{r,s}$ if and only if $G_D^c$ has exactly two connected components; otherwise $G_D^c$ is connected. The key observation is now that the connected components of $G_D^c$ coincide exactly with the elements of a basis of ker$(\delta_0)$. The fact that $\delta_{-1}(\emptyset)=1+2+\dots + c \in C^0(X_D)$ completes the argument.
\end{proof}


\section{Lando cohomology as homology of a simplicial complex}\label{grafoacomplejosection}

In this section we show how to construct a simplicial complex whose homology is equal to the cohomology of the independence complex of a Lando graph $G_D$ up to some shifting. This fact together with Theorem \ref{KeyTheorem} implies that the homology of the simplicial complex determines the extreme Khovanov cohomology of the link represented by $D$.

A key point is the following result by Jonsson \cite{Jonsson}:

\begin{theorem}\emph{\cite[Theorem 3.1]{Jonsson}} \label{jonssongac}
Let $G$ be a bipartite graph with nonempty parts $V$ and $W$. Then there exists a simplicial complex $X_{G,V}$ whose suspension is homotopy equivalent to the independence complex of $G$.
\end{theorem}

In \cite{Jonsson} Jonsson also gave the procedure for constructing the complex $X_{G,V}$. Starting with the bipartite graph $G$, a set $\sigma \subseteq V$ belongs to $X_{G,V}$ if and only if there is a vertex $w \in W$ such that $\sigma \bigcup \{w\}$ is an independent set in $G$. In other words, $\sigma \subseteq V$ is a face of $X_{G,V}$ if and only if $\sigma$ is not adjacent to every $w \in W$.

Recall that the Alexander dual of a simplicial complex $X$ with ground set $V$ is a simplicial complex $X^*$ whose faces are the complements of the nonfaces of $X$. The combinatorial Alexander duality (see for example \cite{StanleyAlexn}) relates the homology and cohomology of a given simplicial complex and its Alexander dual:

\begin{theorem} \label{alexanderduality}
Let $X$ be a simplicial complex with a ground set of size $n$. Then the reduced homology of $X$ in degree $i$ is equal to the reduced cohomology of the dual complex $X^*$ in degree $n-i-3$.
\end{theorem}

As Lando graphs are bipartite, these two results together with the fact that a simplicial complex $X$ and its suspension $\mbox{S}(X)$ have the same homology and cohomology with the indices shifted by one, provide an algorithm for computing the cohomology of the independence complex associated to a Lando graph $G_D$ (or equivalently, the extreme Khovanov cohomology of the link represented by $D$) from the homology of an specific simplicial complex.

\begin{theorem}\label{TheoremResumen}
Let $D$ be a diagram of an oriented link $L$ with $n$ negative crossings. Let $j = j_{\min}(D)$. Let $Y_D = (X_{G,V})^*$, where $G = G_D$ is the Lando graph of $D$, with parts $V$ and $W$. Then
$$
H^{i,j}(L) \approx \widetilde{H}_{|V|-i-1-n}(Y_D).
$$
\end{theorem}

\begin{proof}
Let $Z = X_{G,V}$ hence $Y_D = Z^*$. Then
$$
H^{i+1-n,j}(L)
\approx \widetilde{H}^i(X_D)
\approx \widetilde{H}^i(\mbox{S}(Z))
\approx \widetilde{H}^{i-1}(Z)
\approx \widetilde{H}_{|V|-i-2}(Y_D),
$$

where we have applied Theorem~\ref{KeyTheorem} (recall that $X_D = X_G$ is the independence complex of the Lando graph $G = G_D$), the homotopy equivalence $X_G \approx S(X_{G,V})$ given by Theorem \ref{jonssongac}, the relation between the cohomology of a simplicial complex and its suspension, and finally the combinatorial Alexander duality theorem.
\end{proof}

One can also describe the complex $Y_D$ in terms of $s_AD$, avoiding any reference to the Lando graph $G_D$. Start by coloring the regions of $s_AD$ in a chess fashion. Call an $A$-chord white (black) if it is in a white (black) region. The ground set of $Y_D$ is the set of admissible white arcs of $s_AD$, and a set of admissible white arcs $\sigma$ is a simplex of $Y_D$ if and only if for any admissible black arc there is at least an admissible white arc which is not in $\sigma$ whose ends alternate with the ends of the black arc in the same circle of $s_AD$. Note that there are two different choices when coloring the regions; in order to get the simplest ground set of $Y_D$, choose colors in such a way that white regions contain a lower number of admissible $A$-chords than black regions. 

We are now interested in reversing the process above, namely, starting with any simplicial complex, we will construct a bipartite graph with an associated independence simplicial complex whose cohomology is equal to the homology of the original simplicial complex shifted by some degree. Again, the key point is the Alexander duality theorem together with the following result by Jonsson:

\begin{theorem}\emph{\cite[Theorem 3.2]{Jonsson}}\label{jonsson2}
Let $X$ be a simplicial complex. Then there is a bipartite graph $G$ whose independence complex is homotopically equivalent to the suspension of $X$.
\end{theorem}

The bipartite graph $G$ can be constructed taking as set of vertices the disjoint union of the ground set $V$ of the complex $X$, and the set $M$ of maximal faces of $X$. The edges of $G$ are all pairs $\{v, \mu\}$ such that $v \in V$, $\mu \in M$ and $v \not\in \mu$.

We want to remark that, although Theorem \ref{jonsson2} holds for any simplicial complex, the graph obtained by the above procedure is not necessarily the Lando graph associated to a link diagram. A graph $G$ is said to be realizable if there is a link diagram $D$ such that $G = G_D$ (in \cite{extreme} these graphs were originally called convertible).

We are now ready to construct a link with exactly two non-trivial extreme Khovanov cohomology modules. In the following section this example will be a basic piece to obtain interesting families of $H$-thick knots.


\begin{theorem}\label{Thexample}
There exist oriented link diagrams whose extreme Khovanov cohomology modules are non-trivial for two different values of $i$, that is, $H^{i, j_{\min}}(D)$ is non-trivial for two different values of $i$.
\end{theorem}

\begin{proof}
Although our argument is equally valid for any commutative ring $R$ with unit, just for convenience set $R = \mathbb{Z}$, the ring of integers.

Let $X =\{ \emptyset, 1, 2, 3, 4, 5, 12, 23, 34, 41 \}$ be a simplicial complex with ground set $V=\{1, 2, 3,  4, 5\}$. Its topological realization is the disjoint union of a point and a square, as shown in Figure~\ref{teoej1}, hence its reduced homology is $\widetilde{H}_0(X) \approx \widetilde{H}_1(X) \approx \mathbb{Z}$ (the other homology groups being trivial).\\

Consider now the Alexander dual of $X$, $X^{*} =  \left\{ \emptyset, 1, 2, 3, 4, 5, 12, 13, 14, 15, 23, 24, \right.$ $\left. 25, 34, 35, 45, 123, 124, 134, 135, 234, 245 \right\}$. Note that $|X| + |X^{*}| = 32 = 2^{|V|}$. Applying the combinatorial Alexander duality
leads to
$$
\widetilde{H}_{i}(X) \approx \widetilde{H}^{|V|-i-3}(X^*)=\widetilde{H}^{2-i}(X^*),
$$

which implies that $\widetilde{H}^2(X^*) \approx \widetilde{H}^1(X^*) \approx \mathbb{Z}$, the other groups being trivial.

\begin{figure}
\centering
\includegraphics[width = 11cm]{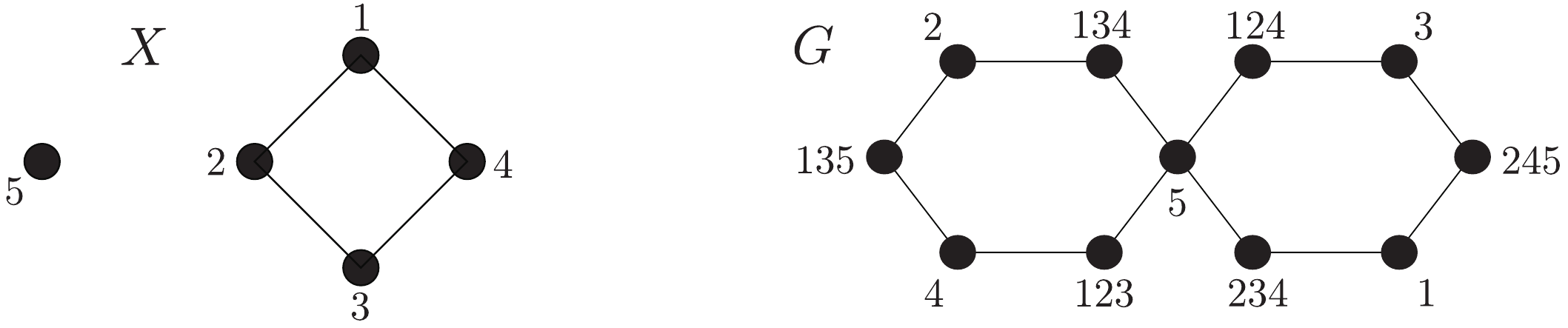}
\caption{\small{The topological realization of the simplicial complex $X$ and the graph $G$. The independence complex of $G$ is homotopy equivalent to the suspension of $X^*$.}}
\label{teoej1}
\end{figure}

Applying Theorem \ref{jonsson2} (and the construction described right after) to the simplicial complex $X^*$ leads to a graph $G$ consisting in two hexagons sharing a common vertex as shown in Figure \ref{teoej1}, whose independence complex $X_G$ is homotopically equivalent to the suspension of $X^*$. In particular,

$$
\widetilde{H}^{i-1} (X^*) \approx \widetilde{H}^i (S(X^*)) \approx \widetilde{H}^i (X_G).
$$

Hence, $\widetilde{H}^3(X_G) \approx \widetilde{H}^2 (X_G) \approx \mathbb{Z}$ are the only non-trivial groups in the reduced cohomology of $X_G$. In fact, as the indices are different from zero, this is still true for the (non-reduced) cohomology, so $H^2(X_G) \approx H^3(X_G) \approx \mathbb{Z}$.

An important point now is the fact that the graph $G$ is realizable. In fact, $G = G_D$ with $D$ being the link diagram in Figure \ref{teoej3}. Indeed, Figure \ref{teoej2} shows the correspondence between the vertices of $G$ and the $A$-chords in $s_AD$.

\begin{figure}
\centering
\includegraphics[width = 11.5cm]{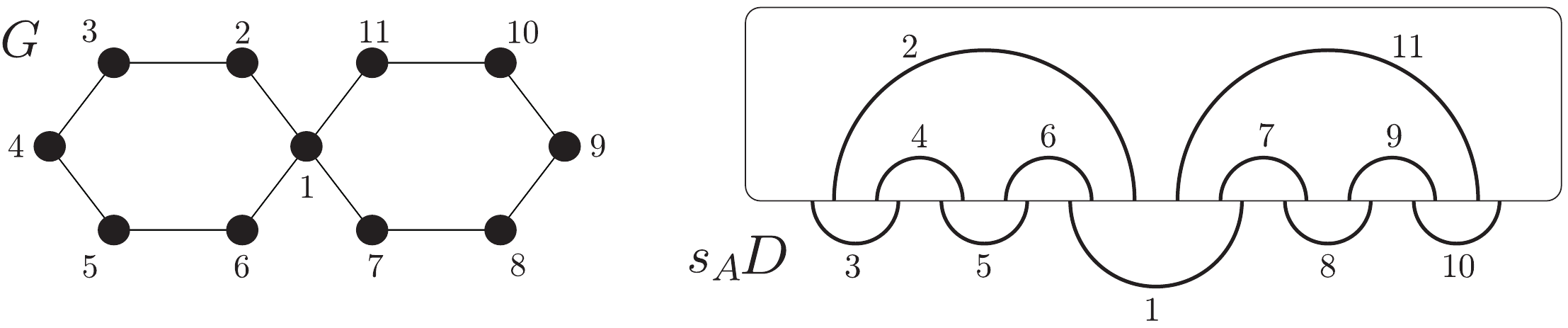}
\caption{\small{The graph $G$ is realizable, the correspondence between its vertices and the $A$-chords in $s_AD$ is shown.}}
\label{teoej2}
\end{figure}

\begin{figure}
\centering
\includegraphics[width = 10cm]{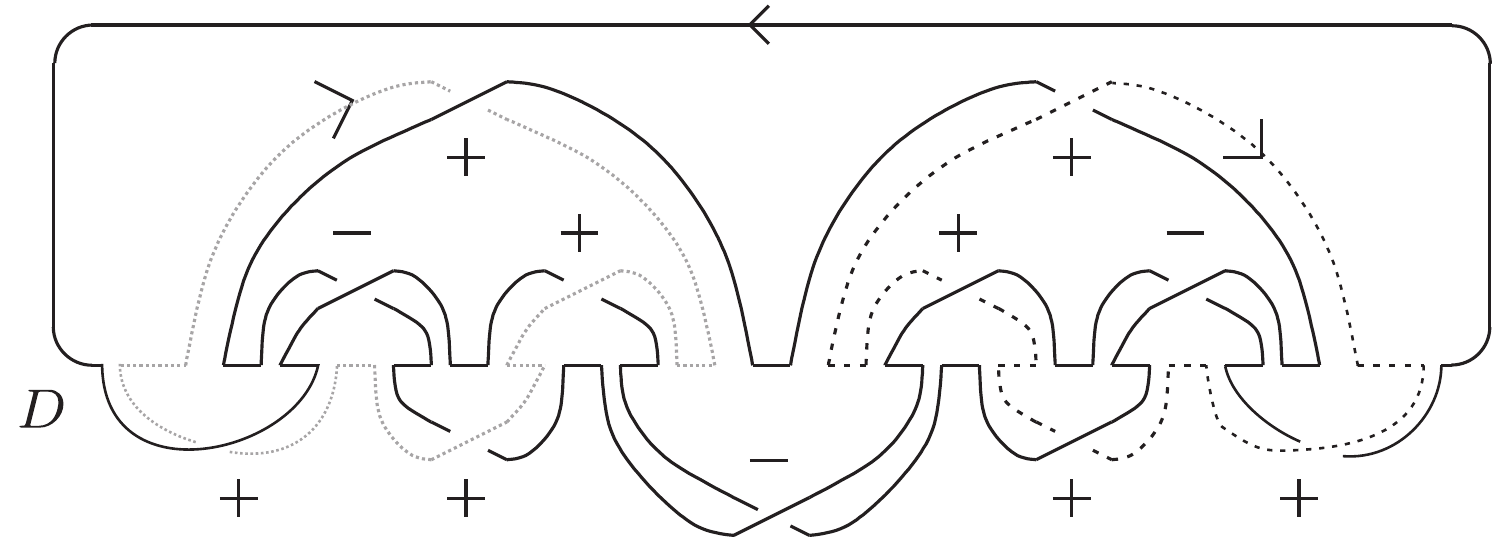}
\caption{\small{The oriented diagram $D$ representing the link $L$.}}
\label{teoej3}
\end{figure}

Consider now the link $L$ represented by the diagram $D$ oriented as shown in Figure~\ref{teoej3}. Then applying Corollary \ref{CorollaryJmin} we get that $j_{\min} = c - 3n - |s_A| = 11 - 3 \cdot 3 - 1 = 1$, and by Theorem~\ref{KeyTheorem} one gets $H^{i,1}(L) \approx \mathbb{Z}$ for $i=0,1$, the other cohomology groups being trivial. This concludes the proof.

\end{proof}

For the example in the previous proof we have checked with computer assistance that the ranks of the chain groups $C^{i} = C^{i}(X_D)$ and differentials $\delta_i$ are
$$1(1)11(10)43(33)73(39)52(12)13(1)1,$$

where the rank of the differentials are parenthesized. This means that the Lando ascendant complex is
$$
0 \, \longrightarrow \, C^{-1} \, \stackrel{\delta_{-1}}{\longrightarrow} \, C^{0} \, \stackrel{\delta_{0}}{\longrightarrow} \, C^{1} \, \stackrel{\delta_{1}}{\longrightarrow} \, C^{2} \, \stackrel{\delta_{2}}{\longrightarrow} \, C^{3} \, \stackrel{\delta_{3}}{\longrightarrow} \, C^{4}
\, \stackrel{\delta_{4}}{\longrightarrow} \, C^{5} \, \stackrel{\delta_{5}}{\longrightarrow} \, 0,
$$

with rk$(C^{-1}) = 1$, rk$(\delta_{-1}) = 1$, rk$(C^{0}) = 11$, rk$(\delta_{0}) = 10$ and so on. Hence
$$
\textnormal{rk}(H^{2}(X_D))=73-39-33=1 \quad \textnormal{ and } \quad \textnormal{rk}(H^{3}(X_D))=52-12-39=1.
$$

\begin{remark}
The proof of Theorem \ref{Thexample} does not work if, for example, we start with the simplicial complex whose topological realization is a point plus a triangle. Although one gets again a graph $G$ such that $X_G$ has two non-trivial cohomology groups, $G$ consists of two hexagons with four common consecutive edges (a total of eight vertices), which is no longer a realizable graph.
\end{remark}


\section{Families of $H$-thick knots}\label{ThickKnotsSection}

Citing Khovanov \cite{patterns}, there are $249$ prime unoriented knots with at most $10$ crossings (not counting mirror images). It is known that for all but $12$ of these knots the non-trivial Khovanov cohomology groups lie on two adjacent diagonals, in a matrix where rows are indexed by $j$ and columns by $i$. Such knots are called $H$-thin. An $H$-thick knot is a knot which is not $H$-thin. For example, any non-split alternating link is $H$-thin, and in the opposite direction, any adequate non-alternating knot is $H$-thick (see \cite{patterns}, Theorem~2.1 and Proposition~5.1).

Up to eleven crossings, there are no knots with more that one non-trivial cohomology group in the rows corresponding to the hypothetical extreme $j_{\max}$ or $j_{\min}$ obtained from the associated diagrams in \cite{atlas}. There are examples which seem to contradict this statement. For example knot $10_{132}$, whose Khovanov cohomology groups are trivial for $j>-1$ and has two non-trivial groups for $j=-1$, but for the diagram of $10_{132}$ taken from \cite{atlas} $j_{\max}(D) = -c + 3p + |s_B| = -10 + 3 \cdot 3 + 2 = 1$. We do not know if there exists a diagram $D$ representing the knot $10_{132}$ with $j_{\max}(D) = -1$.

In this section we show examples of $H$-thick knots having any arbitrary number of non-trivial cohomology groups in the non-trivial row of smallest possible index. More precisely, we will provide a diagram $D$ whose row indexed by $j_{\min}(D)$ is non-trivial, and hence corresponds to the non-trivial row of smallest possible index. Moreover, this row has as many non-trivial cohomology groups as desired. The basic piece in our construction is the link given in the proof of Theorem \ref{Thexample}.

We want to remark that Theorem \ref{KeyTheorem} allows us to compute the extreme Khovanov cohomology of any link diagram $D$ by considering independently each of the circles appearing in $s_AD$, as the non-admissible $A-$chords do not take part in the construction of the simplicial complex $Y_D$ described in Section \ref{grafoacomplejosection}. More precisely, let $D$ be a link diagram and $c_1, \ldots c_n$ the circles of $s_AD$. Write $C_i$ for the circle $c_i$ together with the admissible $A-$chords having both ends in the circle $c_i$, and let $D_i$ be the diagram reconstructed from $C_i$ by reversing the corresponding smoothings. Then, from the construction right after Theorem \ref{TheoremResumen} it follows that $Y_D = Y_{D_1} \ast \ldots \ast Y_{D_n}$, with $\ast$ being the join of simplicial complexes. Recall that the join $X \ast Y$ of two simplicial complexes $X$ and $Y$ is defined as the simplicial complex whose simplices are the disjoint unions of simplices of $X$ and $Y$.

The reduced homology of the join of two simplicial complexes can be computed directly from the reduced homology of each of the complexes, namely
$$
\widetilde{H}_{i}(X \ast Y) = \displaystyle\sum_{r+s=i-1} \widetilde{H}_r(X) \otimes \widetilde{H}_s(Y) \oplus \displaystyle\sum_{r+s=i-2} \textnormal{Tor}(\widetilde{H}_r(X),\widetilde{H}_s(Y)).
$$

Taking copies of the example in the proof of Theorem \ref{Thexample} one obtains a link which, by Theorem \ref{TheoremResumen} and the above formula, has any number of non-trivial extreme Khovanov cohomology groups. Note that, although one could also use the general formula for the Khovanov cohomology of a split link (\cite[Corollary 12]{Khovanov}), we think that our techniques are more useful in order to make computations. Even more, our understanding of extreme Khovanov cohomology in terms of Lando cohomology allows us to slightly modify a link in such a way that one obtains a knot with the same extreme Khovanov cohomology. We explain this construction in detail in Theorem \ref{teothick} and right after it. We need first the following result:

\begin{proposition}\label{ResultCopias}
Let $*_nX$ be the join of $n$ copies of the simplicial complex $X=\{ \emptyset, 1,2,3,4,5,12,23,34,41\}$. Then $\widetilde{H}_i(*_nX)\approx \mathbb{Z}^{\binom{n}{i-n+1}}$ if $n-1 \leq i \leq 2n-1$, and it is trivial otherwise.
\end{proposition}

\begin{proof}
By induction on $n$. In the proof of Theorem~\ref{Thexample} we saw that $\widetilde{H}_0(X)\approx \widetilde{H}_1(X)\approx \mathbb{Z}$, which is the case $n=1$. For $n>1$ we apply the formula for the homology of a join (torsion terms do not appear in any case):
$$
\begin{array}{rcl}
\widetilde{H}_i(*_nX) &  \approx & \bigoplus_{r+s=i-1}\left[ \widetilde{H}_r(*_{n-1}X)\otimes \widetilde{H}_s(X)\right] \\ &&\\
&  \approx & \widetilde{H}_{i-1}(*_{n-1}X) \bigoplus \widetilde{H}_{i-2}(*_{n-1}X)\\ &&\\
 & \approx &
\mathbb{Z}^{\binom{n-1}{(i-1)-(n-1)+1}} \bigoplus \mathbb{Z}^{\binom{n-1}{(i-2)-(n-1)+1}} \\ &&\\
 & \approx &
\mathbb{Z}^{\binom{n-1}{i-n+1}} \bigoplus \mathbb{Z}^{\binom{n-1}{i-n}} \\
&&\\
 & \approx &
\mathbb{Z}^{\binom{n}{i-n+1}}.
\end{array}
$$
\end{proof}

\begin{theorem}\label{teothick}
For every $n > 0$ there exists an oriented knot diagram $D$ with exactly $n+1$ non-trivial extreme Khovanov integer cohomology groups $H^{i, j_{\min}}(D)$.
\end{theorem}

\begin{proof}
Let $L$ be the oriented link represented by the diagram $D$ in Figure \ref{teoej3}. Considering as ground set the chords in the unbounded region of $s_AD$, the associated simplicial complex $Y_D$ is the simplicial complex $X$ appearing in the proof of Theorem \ref{Thexample}, whose topological realization is the disjoint union of a point and a square (Figure \ref{teoej1}). Hence it has two non-trivial reduced homology groups, $\widetilde{H}_0(Y_D) \approx \widetilde{H}_1(Y_D) \approx \mathbb{Z}$.

Now consider the link $L_n$ consisting of the split union of $n$ copies of $L$. It can be represented by $D_n$, the disjoint union of $n$ copies of $D$, so $s_AD_n$ is the disjoint union of $n$ copies of $s_AD$, shown in Figure $\ref{teoej2}$. Hence its associated simplicial complex is $Y_{D_n} = \ast_n Y_{D} = *_nX$, that is, the join of $n$ copies of $X$. Applying Proposition \ref{ResultCopias} to $*_nX$ one gets
$$\widetilde{H}_i(Y_{D_n})  \approx  \mathbb{Z}^{\binom{n}{i-n+1}}$$
for $n-1 \leq i \leq 2n-1$.

This fact together with Theorem \ref{TheoremResumen} shows that the extreme Khovanov cohomology of $L_n$ has $n+1$ non-trivial groups.

Now we will give a knot having the same extreme Khovanov cohomology groups as $L_n$ (the value of $j_{\min}$ changes in general). Starting from the diagram $D$ in Figure \ref{teoej3}, add four crossings, as shown in Figure \ref{trucofi1}, in such a way that the resulting diagram $D'$ has one component. Note that $s_AD'$ is obtained from $s_AD$ by adding two circles with four $A$-chords. Consider now $n$ copies of $D'$ and join them as shown in Figure $\ref{trucofi2}$. The resulting diagram $D_n'$ is a knot diagram. Since $s_AD_n'$ just adds $5n-1$ non-admissible $A$-chords to $s_AD_n$, both diagrams $D_n$ and $D_n'$ share the same Lando graph. Hence $D_n'$ represents a knot having $n+1$ non-trivial groups in its extreme Khovanov cohomology.
\end{proof}

\begin{figure}
\centering
\includegraphics[width = 12cm]{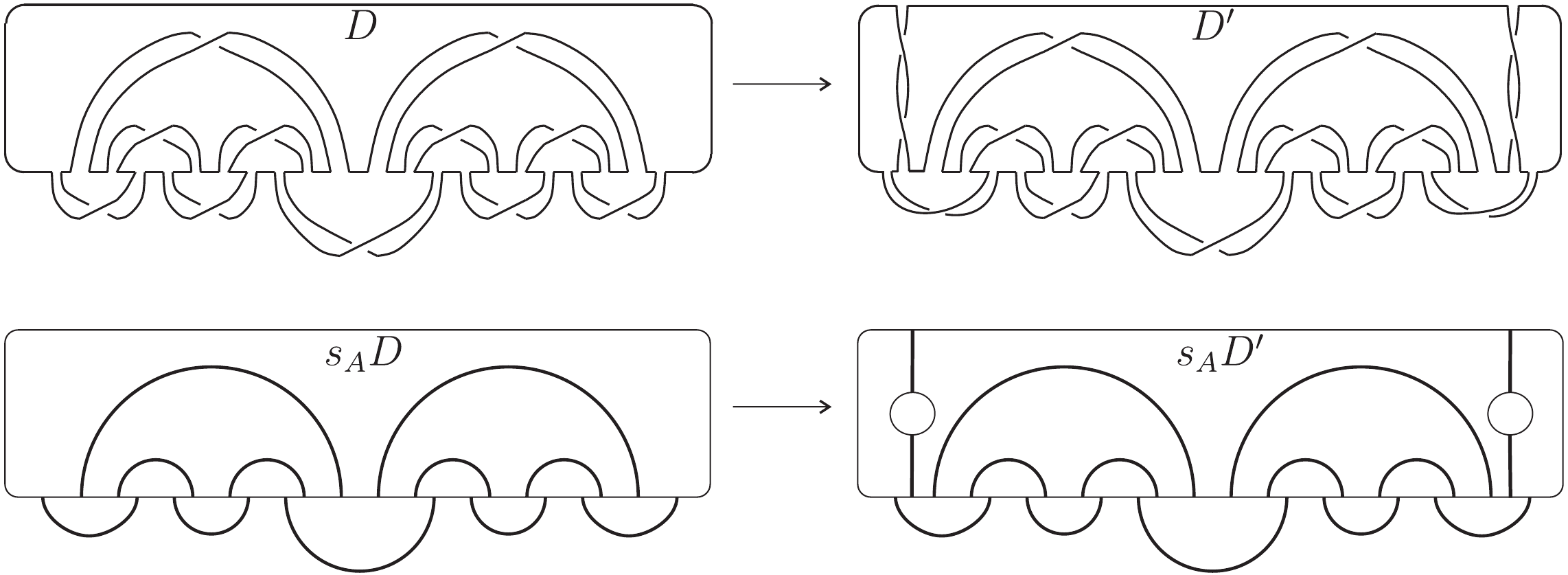}
\caption{\small{The first row shows $D$ and $D'$. The corresponding $s_AD$ and $s_AD'$ are shown in the second row. Note that $D'$ has one component.}}
\label{trucofi1}
\end{figure}

\begin{figure}
\centering
\includegraphics[width = 12.5cm]{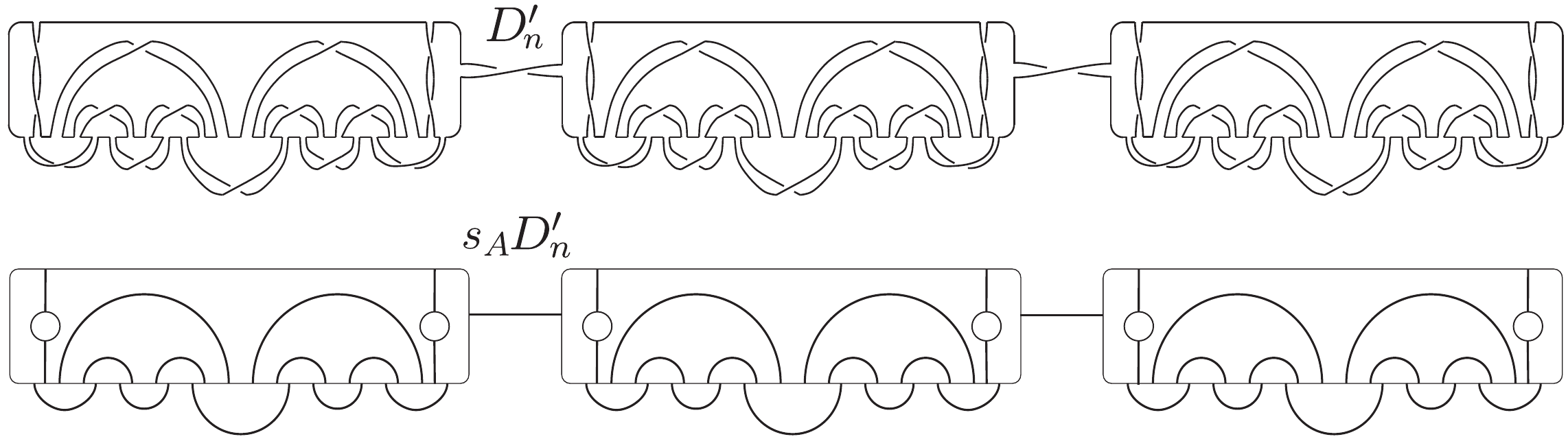}
\caption{\small{$D_n'$ and $s_AD_n'$ are shown for the case $n=3$.}}
\label{trucofi2}
\end{figure}

The proof of Theorem \ref{teothick} provides a family of knots having as many non-trivial extreme Khovanov cohomology groups as desired. These are examples of $H$-thick knots as far of being $H$-thin as desired.

\begin{remark}
Every link $L$ with $\mu$ components can be turned into a knot preserving its extreme Khovanov cohomology ($j_{min}$ can change). One just needs to consider a diagram $D$ of $L$ and add two extra crossings melting two different components into one, as shown in Figure \ref{trucofinal}. Since $s_AD'$ just adds two non-admissible $A$-chords to $s_AD$, both diagrams share the same Lando graph. After repeating this procedure $\mu - 1$ times, the link $L$ is transformed into a knot.
\end{remark}

\begin{figure}
\centering
\includegraphics[width = 12.5cm]{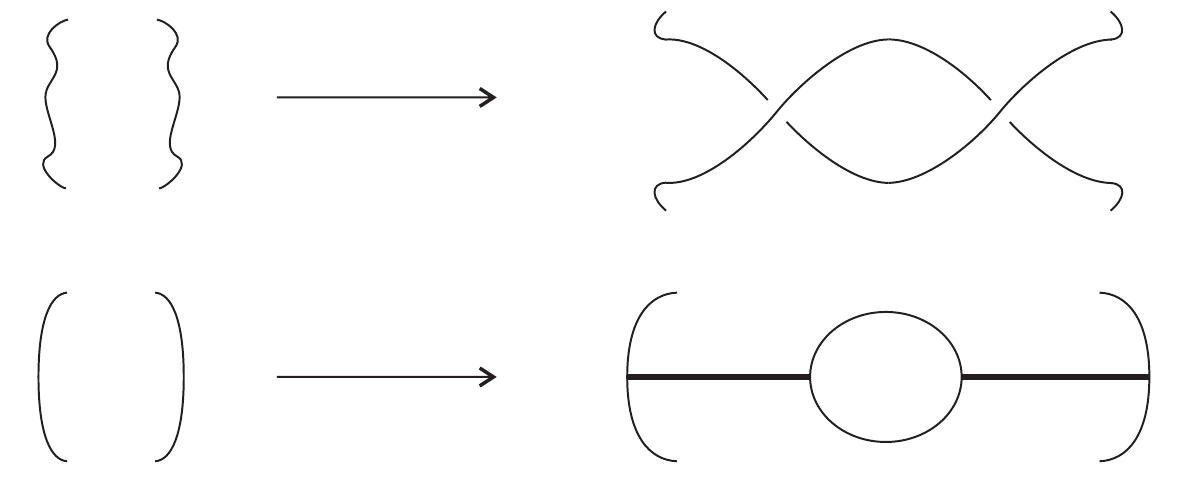}
\caption{\small{This transformation reduces the number of components of the link by one, and the extreme Khovanov cohomology is preserved.}}
\label{trucofinal}
\end{figure}

\end{document}